\documentclass{amsart}

\author{Thomas Morrill}

\usepackage{epstopdf}
\usepackage[bookmarksnumbered,colorlinks,bookmarks,citecolor=blue,linkcolor=blue,urlcolor=blue,breaklinks,linktocpage]{hyperref}
\usepackage[all]{hypcap}
\usepackage{subfigure}

\theoremstyle{plain}
\newtheorem{theorem}{Theorem}[section]
\newtheorem{corollary}[theorem]{Corollary}

\theoremstyle{definition}

\theoremstyle{remark}

\newcommand{\nkshrink}[2]{{#1}_1 < {#1}_2 < \ldots < {#1}_{#2}}
\newcommand{\nklist}[2]{{#1}_1, {#1}_2, \ldots, {#1}_{#2}}
\newcommand{\QQ}{\mathbb Q}
\newcommand{\ZZ}{\mathbb Z}

\begin{document}



\title{Difference Tones in ``Non-Pythagorean'' Scales Based on Logarithms}


\begin{abstract}
  In order to explore tonality outside of the ``Pythagorean'' paradigm of integer ratios, Robert Schneider introduced a musical scale based on the logarithm function.
  We seek to refine Schneider's scale so that the difference tones generated by different degrees of the scale are themselves octave equivalents of notes in the scale.
  In doing so, we prove that a scale which contains all its difference tones in this way must consist solely of integer ratios.
  With this in mind, we present some methods for producing logarithmic scales which contain many, but not all, of the difference tones they generate.
\end{abstract}

\maketitle





{
In 1754, the Italian violinist Giuseppe Tartini described a musical phenomenon which he called \emph{terzi suoni}, or third sounds \cite{discovery, Tartini}.
He found that when two notes were played simultaneously, a third note could be perceived, whose frequency was the difference of the frequencies of the two played notes.
When listening to two tones of frequencies  $f_1 < f_2$, more tones may be perceived in addition to Tartini's \emph{terzi suoni}, whose frequencies are other linear combinations of the frequencies being played \cite{perception}.
Hence, the frequencies $m f_1 + n f_2$, are called the \emph{combination tones} of frequencies $f_1$ and $f_2$, where $n, m \in \ZZ$.

Perception of combination tones in humans has been qualitatively studied \cite{perception}, and combination tones have been empirically measured in string instruments \cite{violins}.
Combination tones are thought to be physically produced by nonlinear resonance in the recieving auditory system, such as the human eardrum \cite{Schwitzgebel, violins}.

Empirical evidence shows that the combination tones $f_2 - f_1$, $2f_1 - f_2$, and $3f_1 - 2f_2$ are the most commonly perceived, and the amplitude of combination tones in general varies depending on the frequencies $f_1$ and $f_2$ and their intensities \cite{perception}.
Being the most commonly perceived combination tone, we will focus exclusively on the \emph{first order difference tone} $f_2 - f_1$.

Since Tartini's discovery, combination tones have been used to develop novel musical scales.
A particularly deep example is Wilson's use of sum-diagonals of the Meru Prastala, perhaps more familiar to the reader as Pascal's triangle, of which Wilson documented 192 recurrent sequences \cite{Burt}.
Another example is Bohlen's use of the combination tone $f_1 + f_2$ in his so-called $833$ cents scale \cite{Bohlen, Smethurst}.

In $2012$, Schneider \cite{Schneider} introduced an infinite scale based over the frequency $f$ in order to explore tonality which does not seek to approximate integer ratios, whose frequencies are given by
\begin{align} \label{log-series}
	f\ln(3), f\ln(4), f\ln(5), \ldots.
\end{align}

Here and throughout the paper, we use $\ln(x)$ to denote the natural logarithm of $x$, {  and $\log_b(x) = \ln(x)/\ln(b)$ to denote the base $b$ logarithm}.
As a consequence of our methodology, we refer to \eqref{log-series} as the \emph{logarithmic series}.
As opposed to the harmonic seres $f, 2f, 3f, \ldots$, the frequencies of the logarithmic series grow closer together as one proceeds further up the series.
A piece using the logarithmic series may be heard in the audio example ``Cantor.wav''.

Schneider also derived a scale which divides the octave into $12$ pitches from the logarithmic series, given by the frequencies $f \ln(4), f \ln(5), \ldots$, $f \ln(16) = 2f \log(4)$.

In private correspondence, Schneider raised the following question:
Which scales contain all of their difference tones up to octave equivalence?
We will show that any such scale must consist solely of integer ratios.
Since the motivation of the logarithmic scale was to explore harmony outside of integer ratios, we offer some methods for constructing logarithmic scales which contain some, but not all of their difference frequencies.

The rest of the paper is laid out as follows.
In Section \ref{methods}, we lay out the methodology and notation necessary to prove our main theorem in Section \ref{proof}.
In Sections \ref{series} and \ref{scales} we offer constructions of frequency series and scales, respectively, which contain many of their difference frequencies.
{Section \ref{composition} describes a chordal technique for these scales which may be of interest to composers.}
{Audio examples may be found at \url{https://github.com/tsmorrill/Non-Pythagorean-Examples}.}
Finally, in Section \ref{end} we give our closing remarks.
}

\section{Methodology} \label{methods}

{
We make some simplifying assumptions on the perception of combination tones.
The first is that combination tones are only generated by the fundamental frequencies of tones, and not their higher partials.
We also assume that only the first order difference tone $f_2 - f_1$ is perceived.
These assumptions are justified based on the empirical evidence regarding perception of combination tones \cite{perception}.

With this in mind, a \emph{frequency} will always be taken to mean a positive real number measured in Hertz.
A \emph{frequency series} is an increasing sequence of frequencies $f_1 < f_2 < \ldots,$ which tends to infinity.
A \emph{scale} is a finite increasing set of frequencies $S = \{f_1, f_2, \ldots, f_n \}$.
The latter definition is compatible with the restriction that a scale consists of frequencies spanning a single octave, but we do not require this for our work.

Transposing a frequency series or scale {  is accomplished by replacing the reference $f_1$ by some other frequency $f_1'$, and replacing each $f_i$ by a $f_i'$ satisfying $f_i : f_1 = f_i' : f_1'$.} {  This} is equivalent to multiplying each of the frequencies $f_i$ by {  the constant $f_1'/f_1$.}

We will therefore \emph{normalize} frequency series and scales by dividing all of their frequencies by $f_1$.
{  This has the effect that all scales and series we present satisfy $f_1 = 1$.}
{  When these frequencies are defined by logarithms, this also has the effect of changing the base of the logarithm.}
{  For example, the scale $S = \{\ln(4), \ln(5) \}$ normalizes to $\{1, \log_4(5)\}$.}
{  Note that the normalization preserves ratios between frequencies, here
\begin{align*}
	\frac{\ln(5)}{\ln(4)} = \frac{\log_4(5)}{1}.
\end{align*}
}
{  From this point, normalized frequencies} will be represented by a closed form expression of $f_i/f_1$, a decimal approximation of $f_i/f_1$, and the interval $f_i:f_1$ measured in cents.
All decimal expansions will be rounded to the third place.
For example, here is the normalization of Schneider's octave-dividing scale.

\begin{align*}
	\begin{matrix}
		\text{Closed Form} & \text{Decimal} & \text{Cents}\\
		\hline
		1 & 1 & 0 \\
		\log_4(5)& 1.161 & 258.388 \\
		\log_4(6) & 1.292 & 444.172 \\
		\log_4(7) & 1.404 & 587.054 \\
		3/2 & 1.500 & 701.955 \\
		\log_4(9) & 1.585 & 797.338 \\
		\log_4(10) & 1.661 & 878.425 \\
		\log_4(11) & 1.730 & 948.642 \\
		\log_4(12) & 1.792 & 1010.35 \\
		\log_4(13) & 1.850 & 1065.236 \\
		\log_4(14) & 1.904 & 1114.547 \\
		\log_4(15) & 1.953 & 1159.225 \\
		2 & 2 & 1200
	\end{matrix}
\end{align*}
}
This may be heard in the audio example ``Schneider Scale.wav''.

\section{Complete Difference Tone Scales} \label{proof}
{
In addition to the terminology of the previous section}, we will call a scale $S$ a \emph{complete difference tone scale} if for every pair $x, y \in S$ with $x > y$, there exists a $z \in S$ and an integer $t$ such that
\begin{align*}
		x - y = 2^t z.
\end{align*}
{
We now prove our main result.}
\begin{theorem} \label{too-much}
	A complete difference tone scale consists solely of rational intervals.
\end{theorem}
{
More {  precisely}, any complete difference scale $S$ must be contained in the set $f \QQ = \{ fr | r \in \QQ \}$ {  for some frequency $f$}.
Thus, any ratio of frequencies in the scale is a rational number.
}

\begin{proof}[Proof of Theorem \ref{too-much}]
{
	Suppose that $S = \{ \nklist{f}{n} \}$ is a complete difference tone scale.
	{Then for each index $i < n$, there is an index $j$ with $1 \leq j \leq n$ and an integer $t_i$} so that $f_n - f_i = 2^{t_i} f_j$.
	We use these relations to define a function $h$ {so that} $h(i) = j$.
	{As this would be undefined for $i=n$, we} additionally define $h(n) = n$.

	We call $i$ a \emph{periodic point under $h$} if $h(h(\cdots(h(s))) = h^{s}(i) = i$ for some $s \geq 1$.
	First, we claim that periodic points under $h$ {  are the indices of} rational multiples of $f_n$.
	This is trivial for the periodic point $i = n$.
	For periodic points $i \neq n$, we may write
	\begin{align*}
		f_n - f_{i_1} &= 2^{t_{i_1}} f_{i_2}\\
		f_n - f_{i_2} &= 2^{t_{i_2}} f_{i_3}\\
		&\vdots\\
		f_n - f_{i_{s-1}} &= 2^{t_{i_s}} f_{i_1},
	\end{align*}
	which implies
	\begin{align*}
		f_n - f_{i_1} = 2^{t_{i_1}}(f_n - 2^{t_{i_2}}(f_n - 2^{t_{i_3}}(f_n - \cdots 2^{t_{i_s}} f_{i_1}))).
	\end{align*}
	By rearranging, we have $\alpha f_{i_1} = \beta f_n$, where $\alpha$ and $\beta$ are rational numbers.

	We call $i$ a \emph{preperiodic point under $h$} if $h^{s}(i)$ is a periodic point under $h$ for some $s \geq 1$.
	We claim that preperiodic points under $h$ {  are also the indices of} rational multiples of $f_n$.
	We have
	\begin{align*}
		f_n - f_{i_1} &= 2^{t_{i_1}} f_{i_2}\\
		f_n - f_{i_2} &= 2^{t_{i_2}} f_{i_3}\\
		&\vdots\\
		f_n - f_{i_{s_1}} &= 2^{t_{i_s}} f_{i_s} = 2^{t_{i_s}} \alpha f_{n},
	\end{align*}
	where $\alpha$ is some rational number.
	This implies that
	\begin{align*}
		f_n - f_{i_1} = 2^{t_{i_1}}(f_n - 2^{t_{i_2}}(f_n - 2^{t_{i_3}}(f_n - \cdots 2^{t_{i_s}} \alpha f_{n}))).
	\end{align*}
	By rearranging, we find that $f_{i_1}$ is a rational multiple of $f_n$.

	Finally, note that each $i$ is either periodic or preperiodic under $h$.
	Thus, every $f_i \in S$ is a rational multiple of $f_n$, which is to say, $S$ consists solely of rational intervals.
	}
\end{proof}

{  We pause to reflect on the consequences of Theorem \ref{too-much}.}
{  Consider its converse:}
\begin{corollary} \label{converse}
	{  A scale which contains an irrational interval is not a complete difference tone scale.}
\end{corollary}
{  Thus, while it is possible to construct a complete difference tone scale from the logarithmic series, the normalization of such a scale would reduce to rational freqeuncies.}
{  Consider the scale $S = \{\ln(16), \ln(32), \ln(64)\}$.}
{  The first order differences of this scale are
\begin{align*}
	\ln(64) - \ln(32) &= \ln(2) = \frac{1}{4}\ln(16)\\
	\ln(64) - \ln(16) &= \ln(4) = \frac{1}{2}\ln(16) \\
	\ln(32) - \ln(16) &= \ln(2) = \frac{1}{4}\ln(16).\\
\end{align*}
We observe that $S$ is a complete difference tone scale.
However, $S$ contains no irrational intervals, which may be seen from its normalization, $S' = \{1, 5/4, 3/2 \}$.
}
{  In keeping with Schneider's stated goal of exploring harmony outside of integer ratios, we will now construct logarithmic series and scales which feature irrational intervals and contain some of their difference frequencies.}
{  However, as a consequence of Corollary \ref{converse}, these scales cannot contain all of their difference frequencies.}

\section{Construction of Logarithmic Frequency Series} \label{series}
{
{  Consider the} difference frequencies generated by a frequency series.
For example, it's easy to check that the harmonic series $f, 2f, 3f, \ldots,$ contains all of its difference frequencies:
Any two frequencies $mf < nf$ have a difference frequency of $(n-m)f$, which also occurs in this series.
However, this is not true of Schneider's logarithmic series $f \log(3), f\log(4), f\log(5)\, \ldots,$ as the first two frequencies have a difference of $f \log(4/3)$.
In general, the difference frequency $f\log(m/n)$ occurs in the logarithmic series if and only if $n/m$ reduces to an integer.
How then can we restrict the logarithmic series such that $m/n$ will always reduce to an integer?

One option is to take logarithms of the factorial function \cite{Factorial}, { which gives what we call the \emph{logarithmic factorial series}.}
\begin{align*}
	\begin{matrix}
		\text{Closed Form} & \text{Decimal} & \text{Cents} \\
		\hline
		1 & 1 & 0 \\
		\log_2(6) & 2.585 & 1644.172 \\
		\log_2(24) & 4.585 & 2636.292 \\
		\log_2(120) & 6.907 & 3345.644 \\
		\log_2(720) & 9.492 & 3896.028 \\
		\log_2(5040) & 12.299 & 4344.592 \\
		\log_2(40320) & 15.299 & 4722.462 \\
		\vdots & \vdots & \vdots
	\end{matrix}
\end{align*}
The differences of this series are {  of the form $f\log( n!/m! )  = f\log((m+1)(m+2) \cdots n)$}, which falls in the logarithmic series.

Another { method } is to take { logarithms of the \emph{primorial numbers}, where each entry $n_k$ is the product of the first $k$ {  prime numbers $p_1, p_2, \ldots p_k$.
Suppose that $x$ is a real number and $p_k$ is the largest prime no greater than $x$.
Then we have
$$
\ln(n_k) = \ln \big( p_1  p_2   \cdots  p_{k} \big) = \sum_{p_i \leq x} \ln(p_i).
$$
which is the definition of the Chebyshev theta function, $\vartheta(x)$.
Hence, we call the corresponding frequency series} the \emph{Chebyshev series}.}
\begin{align*}
	\begin{matrix}
		\text{Closed Form} & \text{Decimal} & \text{Cents} \\
		\hline
		1 & 1 & 0 \\
		\log_2(6) & 2.585 & 1644.172 \\
		\log_2(30) & 4.907 & 2753.771 \\
		\log_2(210) & 7.714 & 3537.03 \\
		\log_2(2310) & 11.174 & 4178.439 \\
		\log_2(30030) & 14.874 & 4673.679 \\
		\log_2(510510) & 18.962 & 5094.009 \\
		\log_2(9699690) & 23.210 & 5443.973 \\
		\vdots & \vdots & \vdots
	\end{matrix}
\end{align*}
{
Similarly to the logarithmic factorial series,} the difference frequencies of this series are of the form $f\ln( p_i p_{i+1} \cdots p_j)$, which falls in the logarithmic series.
However, the difference frequencies of both these subseries fall in the full logarithmic series, not in the subseries themselves.

We could instead choose a sequence of integers $d_1, d_2,$ $\ldots,$ $d_k$ {  and specify that the sequence of difference frequencies is periodic:
\begin{align*}
	f_2 - f_1 &= \ln d_1\\
	f_3 - f_2 &= \ln d_2\\
	&\vdots\\
	f_{k-1} - f_k &= \ln d_k\\
	f_k - f_{k+1} &= \ln d_1\\
	f_{k+2} - f_{k+1} &= \ln d_2\\
	&\vdots
\end{align*}
This correseponds to the frequency series}
\begin{align*}
	\ln(d_1),\
	\ln(d_1d_2) \ldots,
	\ln(d_1 d_2 \cdots d_k),
	 \ln(d_1^2 d_2 \cdots d_k),
	  \ln(d_1^2 d_2^2 \cdots d_k),
	   \ldots.
\end{align*}
{  This series} contains all of its difference frequencies, with the exception of $\ln(d_2)$, $\ln(d_3)$, \ldots, $\ln(d_k)$.
Appending {  the missing frequencies $\ln(d_i)$} may introduce additional difference frequencies depending on their divisibility properties.
For example, with $k=2$ and $d_1 = 3$, $d_2 =5$, we have
\begin{align*}
	\begin{matrix}
		\text{Closed Form} & \text{Decimal} & \text{Cents} \\
		\hline
		1 & 1 & 0 \\
		\log_3(5) & 1.465 & 661.050 \\
		\log_3(15) & 2.465 & 1561.887 \\
		\log_3(45) & 3.465 & 2151.413 \\
		\log_3(225) & 4.930 & 2761.887 \\
		\log_3(675) & 5.930 & 3081.623 \\
		\log_3(3375) & 7.395 & 3463.842 \\
		\log_3(16875) & 8.860 & 3776.747 \\
		\vdots & \vdots & \vdots
	\end{matrix}
\end{align*}

\section{Construction of Logarithmic Scales} \label{scales}
Here we seek to { produce logarithmic scales which contain many, but not all, of their difference tones.
Occasionally, these methods will produce a rational interval.
This was also true for Schneider's original scale: note the perfect fifth between $f\log(4)$ and $f \log(8)$.
Thus, we do not consider the occasional rational interval a flaw of our method, so long as we also produce irrational intervals.
Throughout this section we take $f$ to be an arbitrary reference frequency.

We start with two families of scales which divide the octave.
For the first family,} choose integers $n$ and $m$ so that {  for some positive integer $k$,} $n^{2^k}$ {  is approximately equal to} $m$.
{Then the \emph{root approximation scale} is given by}
\begin{align*}
	\{
	f\log(m),
	f\log(nm), \
	f\log(n^2m),
	\ldots,
	f\log(n^{2^{k-1}}m),
	f\log(n^{2^{2k}}),
	f\log(m^2)
	\}.
\end{align*}
In this scale, many of the difference frequencies take the form $2^sf \log(n)$, which is an octave equivalent of $f\log(n^{2^{k}})$.
{
For example, if we {  approximate $17$ by $2^4$}, then corresponding scale is given by}
\begin{align*}
	\begin{matrix}
		\text{Closed Form} & \text{Decimal} & \text{Cents} \\
		\hline
		1 & 1 & 0 \\
		\log_{17}(34) & 1.245 & 378.889 \\
		\log_{17}(68) & 1.489 & 689.563 \\
		\log_{17}(136) & 1.734 & 952.876 \\
		\log_{17}(256) & 1.957 & 1162.553 \\
		2 & 2 & 1200 \\
	\end{matrix}
\end{align*}
This may be heard in the audio example ``Root Approximation Scale.wav''.
{  Note that if $n^{2^k} = m$, then the scale reduces to}
\begin{align*}
	\{
	f\log(n^{2^k}),
	f\log(n^{2^k + 1}), \
	f\log(n^{2^k + 2}),
	\ldots,
	f\log(n^{2^{2k-1}}),
	f\log(n^{2^{2k}})
	\},
\end{align*}
{  As in Section \ref{proof}, we see that this scale consists exclusively of rational intervals, which runs counter to our goal.}

{
For the second family, we choose a composite integer $N$ with the prime factorization $p_1^{a_1} p_2^{a_2} \cdots p_t^{a_t}$, whose positive {  divisors} we list as
\begin{align*}
	1 = \nkshrink{n}{k} = N.
\end{align*}
The \emph{factorization scale} is given by}
\begin{align*}
	\{ f\ln(1\times N), f\ln(n_2 N), \ldots, f\ln(n_{k-1} N), f\ln(N^2) \}
	\cup
	\{f\ln(p_i^{2^{b_i}}) | 1 \leq i \leq t\},
\end{align*}
where $b_i = \lceil \log_2(\log_{p_i}(N)) \rceil$.
For example, letting $N = 108 = 2^2 3^3$ produces the scale
\begin{align*}
	\begin{matrix}
		\text{Closed Form} & \text{Decimal} & \text{Cents} \\
		\hline
		1 & 1 & 0 \\
		\log_{108}(216) & 1.148 & 239.009 \\
		\log_{108}(256) & 1.184 & 292.882 \\
		\log_{108}(324) & 1.235 & 364.908 \\
		\log_{108}(432) & 1.296 & 448.988 \\
		\log_{108}(648) & 1.383 & 560.961 \\
		\log_{108}(972) & 1.469 & 666.130 \\
		\log_{108}(1296) & 1.531 & 737.054 \\
		\log_{108}(1944) & 1.617 & 832.326 \\
		\log_{108}(2916) & 1.704 & 922.627 \\
		\log_{108}(3888) & 1.765 & 983.956 \\
		\log_{108}(5832) & 1.852 & 1066.863 \\
		\log_{108}(6561) & 1.877 & 1090.220 \\
		2 & 2 & 1200
	\end{matrix}
\end{align*}
This may be heard in the audio example ``Factorization Scale.wav''.
{
Here, the difference frequencies take the form $f \ln(p_1^{b_1} p_2^{b_2} \cdots p_t^{b_t})$, where $-a_i \leq b_i \leq a_i$.
{  If $m = p_1^{b_1} p_2^{b_2} \cdots p_t^{b_t}$ reduces to an integer and satisfies $m = a_j^{2^t}$ for some $j$ and $t$, then the difference tone $\ln(m)$ occurs in the scale, up to octave equivalence.}
{  However,} if $m$ {  does not reduce to an integer}, then $\ln(m)$ {  will not} appear in the logarithmic series, {  let alone the factorization scale}.

This raises the question, can we develop scales whose difference frequencies {  are the logarithms of rational numbers, rather than integers?}
We close with a {  such a family, whose scales} do not contain redundant intervals, and do not span one octave.
{  Non-octave scales are not unheard of in musical practice.}
{  For example, none of Carlos's $\alpha$, $\beta$, and $\gamma$ scales contain the octave \cite{Carlos}.}

Choose bases $\nklist{b}{k}$ and heights $h_{b_1}, h_{b_2}, \ldots, h_{b_k}$.
Let $T$ be the set of rational numbers of the form { $t = b_1^{a_1} b_1^{a_1} \cdots b_k^{a_k}$, where we require $|a_i|\leq h_i$, $\gcd(\nklist{a}{k}) = 1$, and $t > 1$.
The \emph{projective scale} is then given by $S = \{f \log(t) | t \in T\}$.
For example, choosing bases $2, 3$ and heights $h_2 = 2$, $h_3 = 1$ produces the set $T = \{ 4/3, 3/2, 2, 3, 6, 12\}$, and the projective scale is given by}
\begin{align*}
	\begin{matrix}
		\text{Closed Form} & \text{Decimal} & \text{Cents} \\
		\hline
		1 & 1 & 0 \\
		\log_{4/3}(3/2) & 1.409 & 594.123 \\
		\log_{4/3}(2) & 2.409 & 1522.424 \\
		\log_{4/3}(3) & 3.819 & 2319.762 \\
		\log_{4/3}(6) & 6.228 & 3166.596 \\
		\log_{4/3}(12) & 8.638 & 3732.773
	\end{matrix}
\end{align*}
This may be heard in the audio example ``Projective Scale.wav''.
In this example, the difference frequencies take the form $\log_{4/3}(2^i 3^j)$, which are octave equivalents of frequencies in the scale with the exception of
\begin{align*}
	\log_{4/3}(3/2) - \log_{4/3}(4/3) &= \log_{4/3}(9/8)\\
	\log_{4/3}(6) - \log_{4/3}(4/3) &= \log_{4/3}(9/2).
\end{align*}

\section{Logarithmic Composition} \label{composition}

{  It is natural to ask what benefit these scales offer to a composer.}
{  Certainly the main appeal is a systematic method to incorporate difference frequencies while excluding rational intervals from composition.
{  Controlling when rational and irrational intervals occur may be achieved using the factorization of integers into primes.}
{ 
Let $A$ be a positive integer with the unique prime factorization
$$
	A = \prod_{i=1}^k p_i^{a_i}.
$$
}
Because of the property
$$
	\ln(A) = \sum_{i=1}^k a_i \log(p_i),
$$
each positive integer corresponds to a unique pitch set in the logarithmic series,
\begin{align*}
	C_A = \{a_1 \log(p_1), a_2 \log (p_2), \ldots, a_k \log(p_k) \}.
\end{align*}
}
{  So long as $A$ has at least three distinct prime divisors, playing the frequencies of $C_A$ simultaneously will produce a chord of the logarithmic series.}
{  Under this assumption, we call $C_A$ the \emph{factored chord} corresponding to $A$.
Two factored chords $C_{A_1}$ and $C_{A_2}$ relate to each other harmonically depending on which primes occur in both the factorizations of $A_1$ and $A_2$.
A short piece using factored chords may be heard in the audio example ``Factored Chords.wav''.
For example, the factored chords $C_{2016}$ and $C_{4752}$ are given by}
\begin{align*}
	C_{2016} &= \{5 \ln(2), 2 \ln(3),  \ln(7) \}\\
	C_{4752} &= \{4 \ln(2), 3 \ln(3), \ln(11) \}.
\end{align*}

{  Transitioning between $C_{2016}$ and $C_{4752}$ involves two rational intervals:}
\begin{align*}
	5\ln(2) &\mapsto 4\ln(2), \qquad \text{down a major third} \\
	2\ln(3) &\mapsto 3 \ln(3), \qquad \text{up a major fifth} \\
\end{align*}
and one irrational interval:
\begin{align*}
	\ln(7) &\mapsto \ln(11).
\end{align*}
{ 
This reflects the fact that $2016$ and $4752$ are both divisible by $2$ and by $3$.
Recordings demonstrating this technique, as well as recordings of each of the constructions in Section \ref{scales} are available online at \url{https://github.com/tsmorrill/Non-Pythagorean-Examples}.
}

\section{Conclusion} \label{end}
{
Our main result is that complete difference tone scales consist solely of rational intervals.
Hence, a scale {  constructed} using logarithms {  with the intention} to {  feature} irrational intervals cannot be a complete difference tone scale.
With this in mind, we have given three refinements of Schneider's logarithmic series, and several families of scales, parameterized by one or more positive integers, which aim to contain some of their difference frequencies, but not all.

Owing to our methodology, these scales do not account for combination tones generated between higher partials of notes, or combination tones besides the first order difference $f_2 - f_1$.
In his experimental pieces, Schneider primarily uses sine waves, which do not have any partials above the fundamental frequency \cite{Schneider}.
We think it would be interesting to use non-harmonic tones with these logarithmic scales, whose partials fall in the logarithmic series, or one of the refinements given in Section \ref{series}.
}



\section*{Funding}

This work was supported by Australian Research Council Discovery Project DP160100932.

\bibliographystyle{alpha}
\bibliography{MySubmissionBibTexDatabase}

\newcommand{\etalchar}[1]{$^{#1}$}
\begin{thebibliography}{LCC11}

\bibitem[Boh05]{Bohlen}
Heinz Bohlen.
\newblock An 833 cents scale: An experiment on harmony, 2005.

\bibitem[Bur02]{Burt}
Warren Burt.
\newblock Developing and composing with scales based on recurrent sequences.
\newblock {\em Proceedings ACMC}, pages 123--132, 2002.

\bibitem[Car]{Carlos}
Wendy Carlos.
\newblock Three asymmetric divisions of the octave.

\bibitem[Inc19]{Factorial}
OEIS~Foundation Inc.
\newblock The on-line encyclopedia of integer sequences, 2019.

\bibitem[Jon35]{discovery}
Arthur~Taber Jones.
\newblock The discovery of difference tones.
\newblock {\em American Journal of Physics}, 3(2):49--51, 1935.

\bibitem[LCC11]{violins}
Angela Lohri, Sandra Carral, and Vasileios Chatziioannou.
\newblock Combination tones in violins.
\newblock {\em Archives of Acoustics}, 36(4):727--740, 2011.

\bibitem[RP99]{perception}
Rudolf Rasch and Reinier Plomp.
\newblock The perception of musical tones.
\newblock In {\em The Psychology of Music (Second Edition)}, pages 89--112.
  Elsevier, 1999.

\bibitem[S{\etalchar{+}}16]{Smethurst}
Reilly Smethurst et~al.
\newblock Two non-octave tunings by heinz bohlen: A practical proposal.
\newblock In {\em Proceedings of Bridges 2016: Mathematics, Music, Art,
  Architecture, Education, Culture}, pages 519--522. Tessellations Publishing,
  2016.

\bibitem[Sch05]{Schwitzgebel}
Eric Schwitzgebel.
\newblock Difference tone training, 2005.

\bibitem[Sch12]{Schneider}
Robert~P. Schneider.
\newblock A {N}on-{P}ythagorean {M}usical {S}cale {B}ased on {L}ogarithms.
\newblock {\em Proceedings of Bridges: Mathematics, Music, Art, Architecture,
  Culture Conference}, pages 549--552, 2012.

\bibitem[Tar54]{Tartini}
Giuseppe Tartini.
\newblock {\em Trattato di musica secondo la vera scienza dell'armonia}.
\newblock Nella stamperia del Seminario, appresso G. Manfr{\`e}, 1754.

\end{thebibliography}

\end{document}